 \documentclass[12pt,reqno]{amsart}
 \newtheorem{thm}{Theorem}[section]

 \newtheorem{prop}[thm]{Proposition}
 \theoremstyle{definition}
 \newtheorem{defn}[thm]{Definition}
 \newtheorem{rem}{Remark}
 
 \numberwithin{equation}{section}

\usepackage[a4paper,left=30mm,right=30mm,top=30mm,bottom=30mm,marginpar=25mm]{geometry}
\usepackage{amsmath, amssymb}
\usepackage{amsfonts}
\usepackage{amsthm}
\usepackage{amsthm}
\usepackage{mathrsfs}
\usepackage{comment}
\usepackage{color}
\usepackage[normalem]{ulem} %訂正用
\usepackage{enumerate}

\usepackage[nocompress, space]{cite}
\usepackage[displaymath,mathlines]{lineno}
\usepackage{eurosym}
\usepackage[dvips]{graphics}
\usepackage{graphicx}
\usepackage{epsfig}
\usepackage{mathtools}
\usepackage[dvipdfmx]{hyperref}
\usepackage[nocompress, space]{cite}
\usepackage{caption}
\usepackage{subcaption}

\newcommand{\ep}{\varepsilon}

\newcommand{\R}{\mathbb{R}}
\newcommand{\T}{\mathbb{T}}

\newcommand{\Z}{\mathbb{Z}}

\def\ds{\displaystyle}

\newcommand{\Li}{L^{\infty}}
\newcommand{\Lil}{L^{\infty}_{{\rm loc}}}
\newcommand{\Lip}{{\rm Lip\,}}
\newcommand{\Lipl}{{\rm Lip}_{{\rm loc}}}
\newcommand{\W}{W^{1,\infty}}

\theoremstyle{definition}

\numberwithin{equation}{section}

\title[Entropy and viscosity solutions]{On equivalence of entropy and viscosity solutions to degenerate parabolic equations and applications}

\author{Hiroyoshi Mitake}
\address{Graduate School of Mathematical Sciences, University of Tokyo, 3-8-1 Komaba,
Meguroku, Tokyo, Japan 153-8914}
\email{mitake@ms.u-tokyo.ac.jp}

\author{Hiroshi Watanabe}
\address{Department of Science and Technology, Faculty of Science and Technology, Oita University 700 Dannoharu, Oita, Japan 870-1192}
\email{hwatanabe@oita-u.ac.jp}

\thanks{H. M. 
The first author was partially supported by the JSPS grants: 
KAKENHI  
21H04431,
22K03382, 
24K00531, 
H. W 
The second author was partially supported by the JSPS grants: 
KAKENHI 
20K03696, 
21K03312.
}

\keywords{Parabolic-hyperbolic equation, degenerate Hamilton-Jacobi equations, entropy solutions, viscosity solutions.}
\subjclass[2010]{
35K65, %Degenerate parabolic equations
35D30, %Weak solutions 
49L25} %viscosity solutions to PDE
\date{\today}

\begin{document}
\maketitle

%%%%%%%%%%%%%%%%%%%%%%%%%%%%%%%%% Abstract%%%%%%%%%%%%%%%%%%%%%%%%%%%
\begin{abstract}
Here, we consider anisotropic degenerate parabolic-hyperbolic equations and degenerate quasilinear Hamilton-Jacobi equations. 
We prove the equivalence of two notions of entropy and viscosity solutions of two equations, and apply it to obtain a large-time behavior of viscosity solutions to quasilinear Hamilton-Jacobi equations, and entropy solutions to degenerate parabolic-hyperbolic equations in a periodic setting. 
\end{abstract}

%%%%%%%%%%%%%%%%%%%%%%%%%%%%%%%%% Introduction %%%%%%%%%%%%%%%%%%%%%%%%%%%

\section{Introduction}

In this paper we study the Cauchy problems for anisotropic degenerate parabolic-hyperbolic equations of the form 
\begin{equation}\label{CL}\tag*{(CL)}
\left\{ 
\begin{array}{ll}
\displaystyle \partial_{t} u + \partial_{x} f(x,u) = \partial_{x} ( \alpha(x) \partial_{x}u + \partial_{x}\beta(u))
& \text{in}  \ \R \times (0,\infty), \\
u(x,0)=u_{0}(x) 
& \text{in} \ \R, 
\end{array} 
\right.
\end{equation}
and quasilinear Hamilton-Jacobi equations of the form 
\begin{equation}\label{HJ}\tag*{(HJ)}
\left\{ \begin{array}{ll}
\displaystyle 
\partial_{t} v + f(x, \partial_{x}v) = ( \alpha(x) + \beta'(\partial_{x}v)) \partial^2_{x}v 
& 
\text{in}  \ \R \times (0,\infty), \\
v(x,0)=v_{0}(x)& 
\text{in} \ \R, 
\end{array} \right.
\end{equation}
where $u, v:\R\times[0,\infty)\to\R$ are unknown functions. 
Here, $f : \R \times \R \to \R$, $\alpha, \beta, u_0, v_0: \R\to\R$ are given functions and 
\textit{throughout} the paper we impose the assumptions: 
\begin{enumerate}
\item[(A1)]
$f\in\Lipl(\R \times \R)$,  $f(\cdot,u) \in L^{\infty}(\R)$ for all $u\in\R$, and  
\[
\lim_{|u|\to\infty}\inf_{x\in\R}\frac{1}{2}f(x,u)^2-(\|\alpha\|_{\Li} + \|\beta'\|_{\Li}+ 1)|(\partial_x f)(x,u)u|=\infty,  
\]
\item[(A2)] 
$\partial_{x}f, \partial_{x}^{2} f\in L^{1}(\R_{x} ; L^{\infty}_{{\rm loc}}(\R_{u}))$, 
where we rewrite $\R\times\R$ as $\R_x\times\R_u$ to make the variable clear, 
\item[(A3)] 
$\partial_{u}f, \partial_{u} \partial_{x} f \in L^{\infty}(\R_{x} ; L^{\infty}_{{\rm loc}}(\R_{u}))$, 

\item[(A4)]
$\alpha \in W^{1, \infty}(\R)$ is nonnegative, and $\beta \in C^{1}(\R) \cap \Lip(\R)$ is nondecreasing, 

\item[(A5)] 
$u_{0} \in W^{1,\infty}(\R) \cap BV(\R)$, and $v_0(x):= \int_0^x u_0(y)\,dy$. 
\end{enumerate}

\bigskip
Anisotropic degenerate parabolic-hyperbolic equations and degenerate quasilinear Hamilton--Jacobi equations, which 
are natural generalizations of conservation law equations, and first-order Hamilton--Jacobi equations, respectively, 
are studied well in several contexts (see \cite{CK, CP1, CP2, MT, P19, W17}, \cite{CGMT, GMT}, and references therein). 
Formally it is clear to see that if $u$ is a solution to \ref{CL} then $v(x,t):=\int_0^{x}u(y,t)\,dy + \int_{0}^{t} (-f(x,u)(0,\tau) + (\alpha(x) \partial_{x}u + \partial_{x} \beta(u))(0,\tau)) \,d\tau$ 
is a solution to \ref{HJ}. 
Symmetrically, we can formally expect the converse statement, that is, if $v$ is a solution to \ref{HJ}, then $u:=\partial_{x}v$ is a solution to \ref{CL}. 
We \text{call} this relation an equivalence of solutions to \ref{CL} and \ref{HJ}.  
However we notice that since both \ref{CL}, \ref{HJ} are degenerate parabolic equations, one cannot expect the existence of time global classical solutions. It is rather standard to study \ref{CL} and \ref{HJ} by using two notions of weak solutions, 
entropy solutions and viscosity solutions, respectively. 
It turns out that the equivalence of two weak solutions of \ref{CL} and \ref{HJ} is non-trivial. 
Such an equivalence of weak solutions has been justified for the conservation law and first-order Hamilton--Jacobi equations, 
i.e., under the setting $\alpha=\beta\equiv0$ in \ref{CL}, \ref{HJ} in \cite{C, CFN,P20}. 

In this paper we first establish the equivalence of the entropy solution to \ref{CL} and the viscosity solution to \ref{HJ}. 
\begin{thm}\label{thm:main}
Assume that {\rm(A1)--(A5)} hold. 
Let $v \in \Lip(\R \times [0,\infty))$ be the unique viscosity solution to {\rm\ref{HJ}}. 
Then, $u = \partial_{x}v\in C_{{\rm loc}}([0,\infty) ; L^{1}(\R)) \cap L^{\infty}(\R \times (0,\infty))$ is the unique entropy solution to {\rm\ref{CL}}.  
Conversely, let $u\in C_{{\rm loc}}([0,\infty) ; L^{1}(\R)) \cap L^{\infty}(\R \times (0,\infty))
$ be the unique entropy solution to {\rm\ref{CL}}. 
Then, 
there exists $\widehat{C}(t) \in C([0, \infty))$ satisfying $\widehat{C}(0)=0$ such that $v(x,t) := \int_{0}^{x} u(y, t) \,dy + \widehat{C}(t) \in\Lip(\R \times [0,\infty))$  
is the unique viscosity solution to {\rm\ref{HJ}}.  
\end{thm}

Moreover, we apply Theorem \ref{thm:main} to study the large-time behavior of weak solutions to \ref{CL} and \ref{HJ} in the periodic setting with some conditions. 
We add the assumption:  
\begin{enumerate}
\item[(A6)] The functions $u_{0}$, $v_{0}$, $x\mapsto f(x,u)$ and $\alpha$ are $\Z$-periodic for all $u\in\R$. 
\end{enumerate}
We denote by $\T:=\R / \mathbb{Z}$ the $1$-dimensional torus. 
Let $I=[a,b]$ be the connected component of points such that %$H$
$f(x,u) \equiv f(u)$ is affine and $\beta$ is constant in a neighborhood at $u=0$, that is, 
\begin{align}
I=[a,b]
:=
&\text{the maximal connected subset of} \nonumber\\
&\hspace*{24pt}
\{A\subset\tilde{I}\mid A \ \text{is a connected set, and} \ 0\in A\}, \label{def:I}
\\
\tilde{I}:=&
\{u\in\R\mid f(u)=c_1u + c_{2}, \ \beta(u)=c_3 \ \text{for some} \ c_1, c_2, c_{3}\in\R\}. 
\nonumber
\end{align}
It is clear to see that $a\le 0\le b$ from the definition. 
Note that $I$ can be $\{0\}$. 

\begin{thm}\label{thm:large-HJ}
Assume that {\rm(A1)--(A4)}, {\rm(A6)} with 
%$u_0\in W^{1,\infty}(\R)$ and $v_0(x):=\int_0^xu_0(y)\,dy$, 
\begin{enumerate}
\item[{\rm (A5)'}] $u_0\in W^{1,\infty}(\R)$ and $v_0(x):=\int_0^xu_0(y)\,dy$, and 
\item[{\rm(A7)}] $f(x,u) \equiv f(u)$, $\alpha(x)= 0$ for all $(x,u)\in\T\times\R$  
\end{enumerate} 
hold. 
Let $v$ be the viscosity solution to {\rm\ref{HJ}}. 
There exists $(V,d)\in \Lip(\T)\times\R$ such that 
\[
\|v(\cdot,t)+f(0)t-V(\cdot-dt)\|_{\Li(\T)}\to0
\quad\text{as} \ t\to\infty. 
\]
Moreover, $V$ satisfies the condition 
\[
a(y_2-y_1)\le V(y_2)-V(y_1)\le b(y_2-y_1) 
\quad\text{for} \ y_1, y_2\in\R 
\ \text{with} \ y_1<y_2, 
\]
where $I=[a, b]$ is given by \eqref{def:I}, and $d$ is given by $f(u)-f(0)=du$ for $u\in I=[a,b]$. 
In case $I=\{0\}$, we have $V(x)\equiv 0$. 
%\overline{V}$ for some $\overline{V}\in\R$. 
\end{thm}

\begin{thm}\label{thm:large-CL}
Assume that {\rm(A1)--(A4), (A5)', (A6)} and 
\begin{enumerate}
\item[{\rm(A8)}] $u\mapsto f(x,u)$ is strictly convex, $f\in C^2(\T\times\R)$, $\beta(u)=0$ for all $(x,u)\in\T\times\R$. 
%{\color{blue} 
%\item[{\rm(A9)}] $\partial_{x}f(x,0) \in L^{\infty}(\R)$
%（ここでは必要ないですが，Prop 2.5のL1評価を得る際に必要でしょうか？）
%}
\end{enumerate}
hold. 
Let $u$ be the entropy solution to {\rm \ref{CL}}. There exists 
$\widehat{C}(t) \in C([0,\infty))$ satisfying $\widehat{C}(0)=0$ and  $(U,c) \in L^{\infty}(\T)\times\R$ such that 
\begin{equation*}
\lim_{t \to + \infty} \int_{\T} (u(y,t) + \widehat{C}(t) - (U(y) - ct)) \,dy = 0. 
\end{equation*}
\end{thm}

\medskip
Theorems \ref{thm:large-HJ}, \ref{thm:large-CL} are rather straightforward results of 
\cite[Theorem 1.1]{P19} and \cite[Theorem 1.1]{CGMT} by using our equivalence result, Theorem \ref{thm:main}. 
However, to the best of our knowledge, this type of results on the large-time behavior for \ref{CL} and \ref{HJ} have not been studied in the literature. 

We conclude this introduction by giving some related works to \ref{CL} and \ref{HJ}. 
A well-posedness of entropy solutions to \ref{CL} has been established in \cite{CK, CP1, MT, W17} 
under the setting $f(x,u)=f(u)$ and $a(x)=0$ for $(x,u)\in\R^2$ in a multidimensional case. 
A well-posedness of viscosity solutions to \ref{HJ} is rather standard (see \cite{CIL} for instance). 
The large time behavior of periodic solutions to \ref{CL} and \ref{HJ} has been studied in several settings 
(see \cite{CP2, CGMT, GMT, W21}, and references therein).

\medskip
Organization of this paper. 
In Section \ref{sec:pre}, we recall the definitions of entropy solutions and viscosity solutions, and 
some preliminary results. The proof of Theorem \ref{thm:main} is given in Section \ref{sec:equiv}. 
Section \ref{sec:app} is devoted to the study of the large-time behavior for \ref{CL} and \ref{HJ}.

\section{Preliminaries}\label{sec:pre}
We recall the definitions of entropy solutions of \ref{CL}, and viscosity solutions of \ref{HJ}.  
We say that $(\eta,q,r)$ is a convex entropy flux triplet if 
$\eta\in C^2(\R)$ is convex, and  
\begin{equation*}
\partial_{u} q(x,u) = \eta'(u) \partial_{u}f(x,u),\ \ \partial_{u} r(x,u) = \eta'(u) (\alpha(x) + \beta'(u))
\quad\text{for} \ a.e. \ (x,u)\in\R\times\R,   
\end{equation*} 
which is a generalization of the entropy flux pair introduced by Kru\v{z}kov \cite{Kr}. 

\begin{defn}[Entropy solution]
Let $u: \R \times [0,\infty) \to \R$ be a measurable function. 
We call $u$ an \textit{entropy solution} to \ref{CL} if 
\begin{equation}\label{cond:entropy}
u \in L^{\infty}(\R \times (0,\infty)),\quad 
\sqrt{\alpha(x) + \beta'(u)} \partial_{x}u \in L^{2}_{{\rm loc}}(\R\times(0,\infty)), 
\end{equation} 
and 
\begin{align*}
\ds & \int_{0}^{\infty} \int_{\R} \Big\{ \eta(u) \partial_{t}\varphi + q(x,u) \partial_{x}\varphi + r(x,u) \partial_{x}^{2}\varphi - \eta'(u) 
(\partial_{x}f)(x,u) \varphi + (\partial_{x}q)(x,u) \varphi \\
\ds & \hspace{15mm} + (\partial_{x}r)(x,u) \partial_{x}\varphi \Big\} \,dxdt 
%\ds & \hspace{5mm}
 + \int_{\R} \eta(u_{0}) \varphi(x,0)\,dx \\ 
\ds & \hspace{5mm} \ge \int_{0}^{\infty}\int_{\R} \eta''(u) \big( \alpha(x) + \beta'(u) \big) (\partial_{x}u)^{2} \varphi \,dxdt 
\end{align*}
for any $\varphi \in C_{0}^{\infty}(\R \times [0, \infty),[0,\infty))$, 
and any convex entropy flux triplet $(\eta,q,r)$. 
\end{defn}

\begin{defn}[Viscosity sub/super solution]
Let $v:\R \times [0, \infty)\to\R$ be a continuous function. 
{\rm(i) } 
We call $v$ a \textit{viscosity subsolution} to \ref{HJ} if whenever there exist $(x_0,t_0)$ and 
$\varphi\in C^2(\R\times(0,\infty))$ such that 
$v - \varphi$ attains a local maximum at $(x_0,t_0)$, 
\begin{equation*}
\partial_{t}\varphi(x_0,t_0) + f(x_0,\partial_{x} \varphi(x_0,t_0))
\le ( \alpha(x_0) + \beta'(\partial_{x}\varphi(x_0,t_0)) )\partial^2_{x}\varphi(x_0,t_0). 
\end{equation*}
{\rm(ii)} 
We call $v$ a \textit{viscosity supersolution} to \ref{HJ} if whenever there exist $(x_0,t_0)$ and 
$\varphi\in C^2(\R\times(0,\infty))$ such that 
$v - \varphi$ attains a local minimum at $(x_0,t_0)$, 
\begin{equation*}
\partial_{t}\varphi(x_0,t_0) + f(x_0,\partial_{x} \varphi(x_0,t_0))
\ge ( \alpha(x_0) + \beta'(\partial_{x}\varphi(x_0,t_0)) )\partial^2_{x}\varphi(x_0,t_0). 
\end{equation*}
{\rm(iii)} 
A function $v$ is a \textit{viscosity solution} if it is both a viscosity subsolution and a viscosity supersolution.
\end{defn}

We can establish the well-posedness for \ref{CL} with a slight modification of the arguments in \cite{CK} (see also \cite{CP1}).

\medskip
We consider the approximating problems: for $\ep>0$, 
\begin{equation}\label{VCL}\tag*{(VCL)}
\left\{ \begin{array}{ll}
\ds \partial_{t} u^{\varepsilon} + \partial_{x} f(x,u^{\varepsilon}) = \partial_{x}( \alpha(x) \partial_{x}u^{\varepsilon} + \partial_{x}\beta(u^{\varepsilon})) + \varepsilon \partial_{x}^{2}u^{\varepsilon} & \text{in} \ \R \times (0,\infty)
, \\
u^{\varepsilon}(x,0)=u_{0}(x) & \text{in} \ \R, 
\end{array} \right.
\end{equation}
and 
\begin{equation}\label{VHJ}\tag*{(VHJ)}
\left\{ \begin{array}{ll}
\ds \partial_{t} v^{\varepsilon} + f(x,\partial_{x}v^{\varepsilon}) = (\alpha(x) + \beta'(\partial_{x}v^\ep))\partial^2_{x}v^\ep + \varepsilon \partial_{x}^{2} v^{\varepsilon} & \text{in} \ \R \times (0,\infty), \\
v^{\varepsilon}(x,0)=v_{0}(x) & \text{in} \ \R.
\end{array} \right.
\end{equation}

Note that by classical results of parabolic equations (see \cite{LSU} for instance) under 
assumptions (A1)--(A5) with smoothness $f\in C^2(\R^2)$, 
$\alpha, u_0\in C^2(\R)$, $\beta\in C^3(\R)$  \ref{VCL}, \ref{VHJ} have classical solutions. 
Since the constants $C$ in the estimates in Propositions \ref{prop-viscosity}--\ref{prop-entropy} does not depend 
on the smoothness of $f$, $\alpha$, $\beta$, by approximation, we may assume that these are smooth enough.  

%Since the constantC does not depend on the smoothness of H, by approximation, we may assume that H(y, p) is smooth

\begin{prop}\label{prop-viscosity}
There exists $C>0$ which only depends on $f, \|\partial_xv_0\|_{\W}, \|\beta'\|_{\Li}$, and 
$\|\alpha\|_{\Li}$ such that 
\begin{equation*}
\|\partial_{t} v^\ep\|_{\Li}+\|\partial_{x} v^\ep\|_{\Li}\le C
\quad \text{for all} \ \ep>0. 
%|\partial_{t}v^\ep(x,t)| + |\partial_{x}v^\ep(x,t)| < C \quad
%\text{for all} \ 
%(x,t) \in \R \times (0,\infty) \ \text{and} \ \ep\in(0,1). 
\end{equation*}
\end{prop}
\begin{proof}
Setting $M:=\|f(x,\partial_xv_0)\|_{\Li}+\big( \|\alpha\|_{\Li} + \|\beta'\|_{\Li}+1\big)\|\partial^2_{x}v_0\|_{\Li}%+\|g\|_{L^1}
<\infty$, 
we can easily see that 
$v^{\varepsilon}_{\pm}(x,t) = v_{0}(x) \pm Mt$
are a super/sub solution to (VHJ). 
By the comparison principle, %and the maximum principle, 
we get 
\begin{equation*}
|v^{\varepsilon}(x, t+s) - v^{\varepsilon}(x,t)| \le |v^{\varepsilon}(x,s) - v^{\varepsilon}(x,0)| \le Ms 
\ \text{for all} \ x\in\R, t,s>0,  
\end{equation*}
which implies $\|\partial_t v^\ep\|_{\Li}\le M$. 

Next, set $\varphi := \frac{1}{2}(\partial_{x}v^{\varepsilon})^{2}$. 
Differentiating \ref{VHJ} on $x$ and multiplying by $\partial_{x}v^{\varepsilon}$, we obtain 
\begin{equation*}
\begin{array}{l}
\ds \partial_{t} \varphi + (\partial_{x}f)(x,\partial_{x}v^{\ep})\partial_{x}v^{\ep} + \partial_{u}f(x,\partial_{x}v^{\varepsilon})\partial_{x}\varphi \vspace{2mm}\\
\ds = ( \partial_{x}\alpha(x) + \beta''(\partial_{x}v^{\varepsilon})\partial_{x}^{2}v^{\varepsilon}) \partial_{x}\varphi + (\alpha(x) + \beta'(\partial_{x}v^{\varepsilon}) + \varepsilon) (\partial_{x}^{2} \varphi - (\partial_{x}^{2}v^{\varepsilon})^{2})% + g \partial_{x}v^{\varepsilon}
. 
\end{array}
\end{equation*}
Fix any $T>0$, and for any $\delta>0$, set $\varphi^\delta(x,t):=\varphi(x,t)-\delta(1+|x|^2)^{\frac{1}{2}}$. 
Then, 
\begin{multline*}
\partial_{t} \varphi^{\delta} + (\partial_{x}f)(x,\partial_{x}v^{\ep})\partial_{x}v^{\ep} 
%+ (\partial_{u}f)(x,\partial_{x}v^{\varepsilon})\left(\partial_{x}\varphi^{\delta} + \frac{\delta x}{(1+x^2)^{1/2}}\right) 
\\
= 
\big(-\partial_{u}f(x,\partial_{x}v^{\varepsilon}) + \partial_{x}\alpha(x) + \beta''(\partial_{x}v^{\varepsilon}) \partial_{x}^{2}v^{\varepsilon}\big) \left(\partial_{x}\varphi^{\delta}+\frac{\delta x}{(1+x^2)^{1/2}}\right)  \\
+ (\alpha(x) + \beta'(\partial_{x}v^{\varepsilon}) + \varepsilon)\left(\partial_{x}^{2} \varphi^{\delta}+\frac{\delta}{(1+x^2)^{3/2}} - (\partial_{x}^{2}v^{\varepsilon})^{2}\right)% + g \partial_{x}v^{\varepsilon}
. 
\end{multline*}
Take a point $(x_{0}, t_{0}) \in \R \times [0, T]$ so that 
$\varphi^\delta(x_{0},t_{0}) = \max_{(x,t) \in \R \times [0,T]} \varphi^\delta(x,t)$.  
If $t_0=0$, then $\varphi^\delta(x,t)\le \varphi^\delta(x_0,0)\le \frac{1}{2}\|\partial_xv_0\|^2_{%\W
\Li}$. 
Sending $\delta\to0$ yields the conclusion. 

Therefore, we only need to consider the case where $t_0\in(0,T]$. 
In this case, we have 
\begin{align*}
&\big(\alpha(x_{0}) + \beta'(\partial_{x} v^{\varepsilon})(x_{0},t_{0}) + \varepsilon\big) \big(\partial_{x}^{2} v^{\varepsilon}(x_{0},t_{0})\big)^{2} \\
 \le &\, 
 %g(x_{0}) \partial_{x}v^{\varepsilon}(x_{0},t_{0})
 - (\partial_{x}f)(x_{0}, \partial_{x}v^{\ep}(x_{0},t_{0})) \partial_{x}v^{\ep}(x_{0},t_{0}) \\
& +
 \frac{\delta x_0}{(1+x_0^2)^{1/2}} \big(- \partial_{u}f(x_{0}, \partial_{x}v^{\ep}(x_{0},t_{0})) + \partial_{x}\alpha(x_{0}) + \beta''(\partial_{x}v^{\varepsilon})\partial^2_{x}v^{\varepsilon}(x_0,t_0) \big)\\
 &
+ (\alpha(x_{0}) + \beta'(\partial_{x}v^{\varepsilon})(x_0,t_0) + \varepsilon)\frac{\delta}{(1+x_0^2)^{2/3}} \\
=: & %g(x_{0}) \partial_{x}v^{\varepsilon}(x_{0},t_{0})
- (\partial_{x}f)(x_{0}, \partial_{x}v^{\ep}(x_{0},t_{0})) \partial_{x}v^{\ep}(x_{0},t_{0}) + O^\ep(\delta). 
\end{align*}
Multplying by $\alpha(x_{0}) + \beta'(\partial_{x}v^{\varepsilon})(x_{0},t_{0}) + \varepsilon>0$, we obtain
\begin{equation*}
\begin{array}{l}
\big(( \alpha(x_{0}) + \beta'(\partial_{x} v^{\varepsilon})(x_{0},t_{0}) + \varepsilon) \partial_{x}^{2} v^{\varepsilon}(x_{0},t_{0})\big)^{2} \vspace{2mm}\\
\hspace{5mm} + (\alpha(x_{0}) + \beta'(\partial_{x}v^{\varepsilon})(x_{0},t_{0}) + \varepsilon) %g(x_{0})
(\partial_{x}f)(x_{0},\partial_{x}v^{\ep}(x_{0},t_{0})) \partial_{x}v^{\varepsilon}(x_{0},t_{0}) \le O^\ep(\delta). 
\end{array}
\end{equation*}
Note that 
\begin{align*}
& \big((\alpha(x_{0}) + \beta'(\partial_{x} v^{\varepsilon}(x_{0},t_{0})) + \varepsilon) \partial_{x}^{2} v^{\varepsilon}(x_{0},t_{0})\big)^{2} = \big( \partial_{t}v^{\varepsilon}(x_{0},t_{0}) + f(x_{0}, \partial_{x}v^{\varepsilon}(x_{0},t_{0}) )
 \big)^{2} \\ 
%& \ge \frac{1}{2} f(\partial_{x}v^{\varepsilon})^{2}(x_{0},t_{0}) - \left( \partial_{t}v^{\varepsilon}(x_{0},t_{0}) - \int_{0}^{x_{0}} g(\xi)d\xi \right)^{2} 
\ge&\,  \frac{1}{2} f(x_{0}, \partial_{x}v^{\varepsilon}(x_{0},t_{0}))^{2} - C
\end{align*}
for some $C\ge0$. 
Therefore, we have 
\begin{align*}
& \frac{1}{2}f(x_{0}, \partial_{x}v^{\varepsilon}(x_{0},t_{0}))^{2} \\
%& \le (\alpha(x_{0}) + \beta'(\partial_{x}v^{\varepsilon})(x_{0},t_{0}) + \varepsilon) \big( %g(x_{0}) 
%-(\partial_{x}f)(x_{0},\partial_{x}v^{\ep}(x_{0},t_{0})) \partial_{x}v^{\varepsilon}(x_{0},t_{0})+O^\ep(\delta)\big) + C \\
& \le (\|\alpha\|_{\Li} + \|\beta'\|_{\Li}+ 1)|(\partial_{x}f)(x_{0}, \partial_{x}v^{\varepsilon}(x_{0},t_{0}))\partial_{x}v^{\varepsilon}(x_{0},t_{0})|+O^\ep(\delta) + C
%\le C(1+ |\partial_{x}v^{\varepsilon}(x_{0},t_{0})| )
% \\ 
%& = C_{2} |\partial_{x}v^{\varepsilon}(x_{0},t_{0})| + C_{1} \le \max\{C_{1}, C_{2}\}(|\partial_{x}v^{\varepsilon}(x_{0},t_{0})| + 1). 
\end{align*}
for small $\delta, \ep>0$. 
Therefore, sending $\delta\to0$, in light of (A1), we get $\|\partial_{x}v^{\varepsilon}\|_{\Li} \le M$ for some $M\ge0$
 which is independent of $\ep$. 
%Letting $\varepsilon \to 0$, we get the conclusion. 
\end{proof}

\begin{prop}\label{prop-flux}
Let $u^{\varepsilon}$ be a solution to {\rm\ref{VCL}}. Set  
\begin{equation}\label{def:w-ep}
w^{\varepsilon}(x,t) := - f(x,u^{\varepsilon}) + \alpha(x) \partial_{x}u^\ep + \partial_{x} \beta(u^{\varepsilon}) + \varepsilon \partial_{x} u^{\varepsilon}. 
\end{equation}
We have 
\begin{equation}\label{bdd:flux}
\|w^{\varepsilon}\|_{L^{\infty}} \le \|w^{\varepsilon}(\cdot, 0)\|_{L^{\infty}}. 
\end{equation}
Moreover, there exists a positive constant $C$ such that 
\begin{equation}\label{bdd:u-ep}
\|u^{\varepsilon}\|_{L^{\infty}} \le C. 
\end{equation}
\end{prop}

\begin{proof}
Since $\partial_{x} w^{\varepsilon}(x,t) = \partial_{t} u^{\varepsilon}(x,t)$, $w^{\varepsilon}$ satisfies 
\begin{equation*}
\partial_{t}w^{\varepsilon} + \partial_{u}f(x,u^{\varepsilon}) \partial_{x} w^{\varepsilon} = \alpha(x) \partial_{x}^{2}w^\ep + \partial_{x} ( (\beta'(u^{\varepsilon}) + \varepsilon) \partial_{x} w^{\varepsilon}), 
\end{equation*}
which is a uniformly parabolic equation. By the maximum principle, we obtain  
\begin{equation*}
\|w^{\varepsilon}(\cdot, t)\|_{L^{\infty}} \le \|w^{\varepsilon}(\cdot, 0)\|_{L^{\infty}} 
\quad\text{for all} \ t\ge0. 
%\le C_{1}\| -f(u_{0}) + \partial_{x}\beta(u_{0})\|_{L^{\infty}(\R)} + C_{2} \varepsilon \|u_{0}\|_{L^{1}(\R)}. 
\end{equation*}

Setting $v^\ep(x,t):=\int_0^x u^\ep(y,t)\,dy$, we see that $v^\ep$ is a solution to \ref{VHJ}. 
By Proposition \ref{prop-viscosity}, we obtain $\|u^{\varepsilon}\|_{L^{\infty}}=\|v_x^{\ep}\|_{L^{\infty}} \le C$ for 
some $C > 0$. 
\end{proof}

\begin{prop}\label{prop-entropy}
Let $u^{\varepsilon}$ be a solution to {\rm\ref{VCL}}. There exists a subsequence $u^{\varepsilon_{n}}$ and a function $u \in C_{{\rm loc}}([0, \infty) ; L^{1}(\R)) \cap L^{\infty}(\R \times (0,\infty))$ such that 
\begin{equation}\label{conv-u}
u^{\varepsilon_{n}} \to u\ \ \text{ in } C_{{\rm loc}}([0,\infty) ; L^{1}(\R)), 
\end{equation}
and 
\begin{equation}\label{conv-beta}
\alpha(x) \partial_{x} u^{\varepsilon_{n}} + \partial_{x} (\beta(u^{\varepsilon_{n}}) + \varepsilon_{n} u^{\varepsilon_{n}}) \to \alpha(x) \partial_{x}u + \partial_{x} \beta(u) \ \text{ weakly in } L^{2}_{\rm{loc}}(\R \times (0, \infty)) 
\end{equation}
as $n \to \infty$. Moreover, $u$ is an entropy solution to {\rm\ref{CL}} and there exists a positive constant $C$ such that 
\begin{equation}\label{esti:1}
\| - f(x,u) + \alpha(x)\partial_{x}u + \partial_{x}\beta(u)\|_{L^{\infty}(\R)} %+ \|u\|_{L^{\infty}}) 
\le C
\quad\text{for} \  t \in (0, \infty). %Here, $c = \|g\|_{L^{\infty}(\R)}$. 
\end{equation}
\end{prop}
\begin{proof}
We can prove the convergence results in a similar way in \cite{CP1}. 
Indeed, 
by Proposition \ref{prop-flux}, we can take $K>0$ so that $|u^\ep(x,t)|\le K$ for all $(x,t)\in\R\times[0,\infty)$, 
and by \cite[Lemma 7.3]{Coc}, we obtain the following estimate 
\begin{align*}
\int_{\R} |u^{\varepsilon}|\,dx \le 
%\int_{\R} |u_{0}|\,dx + t M_{3} + M_{4} \int_{0}^{t} \int_{\R} |\partial_{x}u^{\varepsilon}|\,dxdt, 
e^{M_{1}t} \int_{\R} |u_{0}|\,dx + M_{2} M_{1}^{-1}(e^{M_{1}t} -1),
\end{align*}
%{\color{blue}（このL1評価について3月20日のノートのような詳細を加えたら如何でしょうか？）}
where $M_{1} := \|\partial_{u}\partial_{x}f\|_{L^{\infty}(\R_{x} ; L^{\infty}([-K,K]))}$, 
$M_{2} := \|\partial_{x}f\|_{L^{1}(\R_{x} ; L^{\infty}([-K,K]))}$%, $M_{4} := \|\partial_{u}f\|_{L^{\infty}(\R_{x} ; L^{\infty}([-K,K]))}$
. 
In fact, we can compute 
\begin{align*}
& \frac{d}{dt} \int_{\R} |u^{\varepsilon}| dx = \int_{\R} {\rm sgn}(u^{\varepsilon})(- \partial_{x}f(x,u^{\varepsilon}) + \partial_{x}(\alpha(x)\partial_{x}u^{\varepsilon} + \partial_{x}\beta_{\varepsilon}(u^{\varepsilon}))) dx \\[1ex] 
& \le - \int_{\R} {\rm sgn}(u^{\varepsilon}) (\partial_{x}f)(x,u^{\varepsilon}) dx - \int_{\R} \partial_{x} \left( \int_{0}^{u^{\varepsilon}} {\rm sgn}(s) (\partial_{u}f)(x,s)ds \right)dx \\[1ex]
& \hspace{5mm} + \int_{\R} \left( \int_{0}^{u^{\varepsilon}} {\rm sgn}(s) (\partial_{x} \partial_{u}f)(x,s)ds \right)dx \\[1ex]
& \le M_{2} + M_{1} \int_{\R}|u^{\varepsilon}|dx. 
\end{align*}
By Gronwall's inequality, we can get the desired $L^{1}$-estimate.

Also, by \cite[Theorem 2.29]{MNRR}, \cite[Lemma 3.2]{KU} and \cite[Theorem 2.9]{W17}, respectively, for all $h>0$, we obtain 
\begin{align*}
&
\int_{\R} |\partial_{x} u^{\varepsilon}|\, dx \le e^{M_{1}t} \left( \int_{\R} |\partial_{x}u_{0}|\,dx + M_{3}t \right), \\
& 
\int_{\R} |u^{\varepsilon}(x, t+h) - u^{\varepsilon}(x,t)| dx \le M_{4} \nu_{T}(h), \\
& 
\int_{|x| > R}|u^{\varepsilon}(\cdot,t)|dx \to 0  \ \text{ uniformly for } u^{\varepsilon}  \ \text{ as } R \to \infty,  
\end{align*}
where 
$M_{3} := \|\partial_{x}^{2}f\|_{L^{1}(\R_{x} ; L^{\infty}([-K,K]))}$, 
$M_{4}$ is a positive constant which depends on 
$M_{i}$ $(i = 1,\ldots , 3)$, $\|\alpha\|_{L^{\infty}}$, $\|\beta'\|_{L^{\infty}}$, $T$, $K$, 
$\|u_{0}\|_{BV}:=\|u_0\|_{L^1}+\|\partial_xu_0\|_{L^1}$, and $\nu_{T}: [0,\infty) \to [0,\infty)$ is a modulus of continuity. 
Then, the Kolmogorov compactness theorem implies that there exists a subsequence $u^{\varepsilon_{n}}$ and $u \in C_{{\rm loc}}([0,\infty); L^{1}(\R))$ such that (\ref{conv-u}) holds. Combining this result with $L^{2}$ estimate for $\alpha(x) \partial_{x} u^\ep + \partial_{x} (\beta(u^{\varepsilon}) + \varepsilon u^{\varepsilon})$ (see \cite[Lemma 3.7]{W17}), we can see that the limit function $u$ is an entropy solution to (CL).

We finally give a proof of \eqref{esti:1} here. 
Let $w^\ep$ be the function defined by \eqref{def:w-ep}. 
By Proposition \ref{prop-flux}, $\|w^\ep\|_{L^{\infty}}$ is bounded uniformly for $\ep>0$. 
Thus, by taking a subsequence if necessary there exists $w \in L^{\infty}(\R \times (0,\infty))$ such that 
\begin{equation*}
w^{\varepsilon} \to w \ \ \mbox{weakly-}\ast \mbox{ in } L^{\infty}(\R \times (0,\infty)) 
\quad\text{as} \ \ep\to0. 
\end{equation*}
Moreover, by \eqref{bdd:u-ep}, \eqref{conv-u}, and \eqref{conv-beta} we obtain 
\begin{equation*}
w^{\varepsilon} \to - f(x,u) + \alpha(x)\partial_{x}u + \partial_{x} \beta(u) \ \ \mbox{ weakly in } L^{2}_{{\rm loc}}(\R \times (0,\infty)) 
\ \text{as} \ \ep\to0, 
\end{equation*}
which implies $w = - f(x,u) + \alpha(x)\partial_{x}u + \partial_{x}\beta(u)$ for a.e. in $\R \times (0,\infty)$. 
%, and 
%\eqref{bdd:u-ep} for a positive constant which only depends on $\|f(x,u_{0})\|_{L^{\infty}}$, $\|a\|_{L^{\infty}}$, $\|\beta'\|_{L^{\infty}}$ and $\|\partial_{x}u_{0}\|_{L^{\infty}}$. 
\end{proof}

\begin{rem}\label{rem:1}
Assumptions (A2), (A3) are used to obtain Proposition \ref{prop-entropy} and the uniqueness of entropy solutions to \ref{CL}. 
When $f(x,u)\equiv f(u)$, assumptions (A4), (A5) can be reduced to $\beta\in\Lipl(\R)$, $u_{0} = \partial_{x}v_{0}\in L^{\infty}(\R)$ rather easily. 
Indeed, since constants are solutions to \ref{CL}, by using the comparison principle, we can easily get the estimates of $\|\partial_x v\|_{\Li}$ and $\|u\|_{\Li}$ 
under the assumption $u_{0} = \partial_{x}v_{0}\in L^{\infty}(\R)$. 
Therefore, by a similar way to that in \cite{MT}, we can obtain the existence and uniqueness of entropy solutions under $u_{0} \in L^{\infty}(\R)$. 
However, since we need (\ref{bdd:flux}), we assume that $u_{0} \in W^{1,\infty}(\R)$ in this case.

\end{rem}

\section{Equivalence}\label{sec:equiv}
It is now ready to prove Theorem \ref{thm:main}. 
\begin{proof}[{\rm Proof of Theorem \ref{thm:main}}]
First, let $v \in \Lip(\R \times [0,\infty))$ be the unique viscosity solution to \ref{HJ}. 
Let $v^{\varepsilon}$ be the classical solution to \ref{VHJ}. 
By Proposition \ref{prop-viscosity}, and the uniqueness of viscosity solutions to \ref{HJ}, we have 
$v^\ep\to v$ locally uniformly on $\R\times[0,\infty)$ as $\ep\to0$. 
For any $\varphi \in C_{0}^{\infty}(\R \times [0,\infty))$, 
\begin{equation*}
\int_{0}^{\infty} \int_{\R} \partial_{x}v^{\varepsilon} \varphi \, dxdt 
= -\int_{0}^{\infty}\int_{\R} v^{\varepsilon} \partial_{x} \varphi \,dxdt 
\to - \int_{0}^{\infty}\int_{\R} v \partial_{x} \varphi \, dxdt 
= \int_{0}^{\infty}\int_{\R} \partial_{x} v  \varphi \, dxdt 
\end{equation*}
as $\ep\to0$. 
Letting $u^{\varepsilon} := \partial_{x} v^{\varepsilon}$, 
by a direct computation, we see that $u^{\varepsilon}$ is a solution to \ref{VCL}. 
By Proposition \ref{prop-entropy}, 
\begin{equation*}
u^{\varepsilon} \to u \ \ \mbox{ in } %L^{1}_{loc}(\R \times (0,T))
C_{{\rm loc}}([0,\infty); L^{1}(\R))  
\quad \text{as} \ \varepsilon \to 0, 
\end{equation*}
where $u$ is the unique entropy solution to \ref{CL}. 
Hence, 
\begin{equation*}
\int_{0}^{\infty} \int_{\R} u \varphi \,dxdt = \lim_{\varepsilon \to + 0} \int_{0}^{\infty}\int_{\R} u^{\varepsilon} \varphi \,dxdt 
= 
\lim_{\varepsilon \to + 0} \int_{0}^{\infty}\int_{\R} \partial_xv^\ep \varphi \,dxdt 
=
\int_{0}^{\infty} \int_{\R} \partial_{x} v \varphi\, dxdt, 
\end{equation*}
which implies that $u = \partial_{x}v$ a.e. in $\R \times [0,\infty)$. 

\medskip

Next, let  $u$ be an entropy solution to \ref{CL}. 
There exists a solution $u^{\varepsilon}$ to \ref{VCL} such that $u^{\varepsilon} \to u$ in $C_{{\rm loc}}([0,\infty); L^{1}(\R))$. 
Set $v^{\varepsilon}(x,t) := \int_{0}^{x}u^{\varepsilon}(y,t)dy + \int_{0}^{t} (-f(x,u^{\varepsilon})(0,\tau) + ( \alpha(x) \partial_{x} u^{\varepsilon} + \partial_{x}\beta_{\varepsilon}(u^{\varepsilon}))(0,\tau))\, d\tau$. 
Then, $v^{\varepsilon}$ is a classical solution to \ref{VHJ}. 
Moreover, there exists a viscosity solution $v$ such that $v^{\varepsilon} \to v$ locally uniformly on $\R \times [0,\infty)$. 
On the other hand, it also follows that $\int_{0}^{x} u^{\varepsilon}(y,t)\,dy \to \int_{0}^{x} u(y,t)\, dy$ locally uniformly on $\R \times [0,\infty)$ by (\ref{conv-u}). 
Hence, there exists $\widehat{C} \in C([0, \infty))$ such that 
the convergence $\int_{0}^{t} (-f(x,u^{\varepsilon})(0,\tau) + ( \alpha(x) \partial_{x}u^{\varepsilon} + \partial_{x}\beta_{\varepsilon}(u^{\varepsilon}))(0,\tau))\, d\tau \to \widehat{C}(t)$ locally uniformly on $[0,\infty)$ holds, which implies $\widehat{C}(0) = 0$. 
Noting that $u \in \Lil(\R \times (0,\infty))$ is a weak solution to \ref{CL} satisfying $f(x,u) - \alpha(x)\partial_{x}u - \partial_{x}\beta(u) \in L^{\infty}_{{\rm loc}}(\R\times(0,\infty))$. 
Setting $v(x,t) := \int_{0}^{x} u(y,t)\,dy + \widehat{C}(t)$, we first show that $v\in\W_{{\rm loc}}(\R \times (0,\infty))$ is a solution to \ref{HJ} a.e. in $\R \times (0,\infty)$. 
Since $u \in L^{\infty}_{{\rm loc}}(\R \times (0,\infty))$, there exists $A \subset (0,\infty)$ with $\mathscr{L}^{1}(A) = 0$ such that for all $t \in (0,\infty) \backslash A$, $u$ is defined a.e. on $\R$ and $u(\cdot, t) \in L^{\infty}_{{\rm loc}}(\R)$. Then, $v(\cdot, t) \in W^{1, \infty}_{{\rm loc}}(\R)$ for such values of $t$. Moreover, for all $t \in (0,\infty) \backslash A$, $\varphi \in C_{0}^{\infty}(\R \times (0,\infty))$, 
$\int_{\R} v(x,t) \partial_{x} \varphi(x,t)\,dx = - \int_{\R} u(x,t) \varphi(x,t)\,dx$. 
Integrating this on $t$ yields that $u = \partial_{x}v$ a.e. in $\R \times (0,\infty)$. 
Since $u$ is a weak solution, we have for all $\varphi \in C_{0}^{\infty}(\R \times (0,\infty))$
\begin{align*}
& \int_{0}^{\infty} \int_{\R} (f(x,u) - \alpha(x) \partial_{x}u - \partial_{x}\beta(u)) \partial_{x} \varphi\,dxdt 
 = - \int_{0}^{\infty} \int_{\R} %(
u \partial_{t}\varphi% + g \varphi)
\, dxdt \\
=&\, 
\int_{0}^{\infty} \int_{\R} v \partial_{x}\partial_{t} \varphi %G \partial_{x}\varphi 
\,dxdt. 
\end{align*}

Since $f(x,u) - \alpha(x) \partial_{x}u - \partial_{x}\beta(u) \in L^{\infty}_{{\rm loc}}(\R \times (0,\infty))$, we have $\partial_{t} v \in L^{\infty}_{{\rm loc}}(\R \times (0, \infty))$ and, 
\begin{equation*}
\partial_{t} v = - f(x,u) + \alpha(x) \partial_{x}u + \partial_{x} \beta(u) = - f(x,\partial_{x}v) + \alpha(x) \partial_{x}^{2}v + \partial_{x} \beta(\partial_{x} v)
\quad a.e. \ \text{in} \  \R \times (0,\infty). 
\end{equation*}
Consequently, we see that $v \in W^{1, \infty}_{{\rm loc}}(\R \times (0,\infty))$ and $v$ is a solution to \ref{HJ} a.e. in $\R \times (0,\infty)$. 

\medskip
Finally, letting $u \in C_{{\rm loc}}([0,\infty) ; L^{1}(\R))$ is the unique entropy solution to \ref{CL} we prove that $v:=\int_{0}^{x} u(y,t)\,dy 
+ \widehat{C}(t)$ is a viscosity solution. 
Note that $v(\cdot, t) \in AC(\R)$ and $\partial_{x} v = u$ a.e. in $\R \times (0,\infty)$ for all $t \in (0,\infty)$. 
Moreover, $v$ is a solution to \ref{HJ} a.e. in $\R \times (0,\infty)$. 
It is easy to see that 
\begin{equation*}
\lim_{t \to 0} |v(x,t) - v_{0}(x)| \le \lim_{t \to 0} \left( \int_{0}^{x} |u(\xi, t) - u_{0}(\xi)| d\xi + |\widehat{C}(t)| \right) = 0. 
\end{equation*}

Suppose that $v$ is not a viscosity solution to \ref{HJ}. 
Let $\overline{v} \in %W^{1,\infty}(\R \times (0,\infty))
C(\R \times [0, \infty))$ be the unique viscosity solution to \ref{HJ}. 
Then, $\overline{u}:=\partial_{x} \overline{v}$ is an entropy solution to \ref{CL}. 
Moreover, the uniqueness of entropy solutions implies that for all $\varphi \in C_{0}^{\infty}(\R \times (0,\infty))$, 
\begin{equation*}
\int_{0}^{\infty}\int_{\R} (v - \overline{v}) \partial_{x}\varphi \,dxdt = - \int_{0}^{\infty}\int_{\R} (\partial_{x} v - \partial_{x} \overline{v}) \varphi \,dxdt = -  \int_{0}^{\infty}\int_{\R} (u - \overline{u}) \varphi \,dxdt = 0.
\end{equation*}
By the arbitrariness of $\varphi$, we obtain $v = \overline{v}$ a.e. in $\R \times (0,\infty)$. 
\end{proof}

\begin{rem}
We can easily check that when $f(x,u)=f(u)$ %and $\alpha(x)=0$
 for all $(x,u)\in\R\times\R$, 
(A5) can be reduced to (A5)'
%the condition $u_0\in W^{1,\infty}(\R)$ and $v_0(x):=\int_0^xu_0(y)\,dy$ 
to obtain Theorem \ref{thm:main} (cf. Remark \ref{rem:1}). 

On the other hand, under the additional assumption $\partial_{x}f(x,0) \in L^{\infty}(\R)$, we can get another $L^{\infty}$-estimate 
\begin{equation*}
\| u(\cdot,t) \|_{L^{\infty}} \le (\| u_{0} \|_{L^{\infty}} + M_{5}t)e^{M_{1}t}, 
\end{equation*}
for $t > 0$, where $M_{5} := \| \partial_{x}f(\cdot,0) \|_{L^{\infty}}$ by \cite[(4.6)]{Kr}. 
Then, we can remove $BV(\R)$ in (A5). 
At first, we can get the well-posedness for (CL) under $u_{0} \in L^{1}(\R) \cap L^{\infty}(\R)$ along \cite[Theorem 1.2]{CP1}. Next, we can remove $L^{1}(\R)$ using the method in \cite[Theorem 9]{MT}. In fact, using the cut-off function 
\begin{equation*}
\begin{array}{c}
\zeta_{n}(x) = 1\ \ \mbox{ in } [-n,n],\ \ \zeta_{n}(x)=0\ \ \mbox{ in } \R\setminus [-n-1,n+1], \\[1ex] 
\zeta_{n}(x) = x + n +1 \ \ \mbox{ in } [-n-1,-n],\ \ \zeta_{n}(x) = -x +n+1\ \ \mbox{ in } [n,n+1], 
\end{array}
\end{equation*}
we set $u_{0}^{m,n}(x) := u_{0}^{+}(x) \zeta_{n}(x) + u_{0}^{-}(x) \zeta_{m}(x)$ for $u_{0} \in W^{1,\infty}(\R)$. By $u_{0}^{m,n} \in W^{1,\infty}(\R) \cap L^{1}(\R)$, we can get the entropy solution $u^{m,n}$ to (CL) with $u_{0}^{m,n}$. Moreover, $u^{m,n}$ is monotone with respect to $m$, $n$ by the comparison principle. In addition, we can get 
\begin{equation*}
\| u^{m,n}(\cdot,t) \|_{L^{\infty}} \le (\| u_{0}^{m,n} \|_{L^{\infty}} + M_{5}t)e^{M_{1}t} \le (\| u_{0} \|_{L^{\infty}} + M_{5}t)e^{M_{1}t}, 
\end{equation*}
for $t > 0$, and therefore $u^{m,n}$ is bounded independent of $m$, $n$. Then, $u^{m,n}$ converges as $m \to \infty$ and $n \to \infty$ respectively, and the limit function is the entropy solution to (CL) with $u_{0} \in W^{1,\infty}(\R)$ in a similar way in \cite[Theorem 9]{MT}. 
%The above estimate allows us to use the method in \cite{MT}, then (A5) can be reduced (A5)'. 
Hence, (A5) can be reduced (A5)' under $\partial_{x}f(x,0) \in L^{\infty}(\R)$ to obtain Theorem \ref{thm:main}.
\end{rem}

\section{Applications}\label{sec:app}
In this section, we \textit{always} assume (A6). 
%\begin{enumerate}
%\item[(A6)] The functions $u_{0}$, $v_{0}$, $x\mapsto f(x,u)$ and $a$ are $\Z$-periodic for all $u\in\R$. 
%\end{enumerate}
%We denote by $\T$ the $1$-dimensional torus $\R / \mathbb{Z}$. 

\subsection{Large-time behavior for \ref{HJ}}
In this subsection we assume (A7), 
therefore \ref{HJ} turns out to be the equation of the form
\begin{equation}\label{eq:CL2}
\partial_{t}v + f(\partial_{x} v)= \beta'(\partial_{x} v)\partial^2_{x} v
\quad \text{in} \ \T\times(0,\infty).
\end{equation}

We recall a result on the large time behavior for degenerate parabolic-hyperbolic equations in \cite{P19}. 
\begin{thm}[{\cite[Theorem 1.1, Corollary 1.1]{P19}}]\label{thm:panov}
Assume that {\rm (A1)--(A4)}, {\rm (A5)'}, {\rm (A6)} and {\rm (A7)} %and $u_{0} \in W^{1,\infty}(\R)$ and $v_{0}(x) := \int_{0}^{x} u_{0}(y)\, dy$
 hold. 
Let $u$ be the entropy solution to {\rm \ref{CL}}. There exists $(U,d) \in L^{\infty}(\T)\times\R$ such that 
\begin{equation*}
{\rm ess}\lim_{t \to + \infty} \|u(\cdot,t) - U(\cdot - dt)\|_{L^1(\T)} = 0 
\end{equation*}
with $\int_{\T} U(y)\, dy = %\overline{u}_0 :=
0 = \int_{\T} u_{0}(x)\, dx$, 
$U$ is a periodic function,
and the functions $f(u) - du$ and $\beta(u)$ are constant on 
$[ \mbox{{\rm ess}}\inf_{\T} U,  \mbox{{\rm ess}}\sup_{\T} U]$. 
%$[\check{w}, \hat{w}]$, where $\check{w} = \mbox{{\rm ess}}\inf w(y)$, $\hat{w} = \mbox{{\rm ess}}\sup w(y)$). 
Moreover, 
if $f(u) - du$ and $\beta(u)$ are not constant simultaneously in any neighborhood of %$\overline{u}_0$ 
$0$ 
for any $d \in \R$, then $U\equiv %\overline{u}_0
0$ on $\T$. 
\end{thm}

%\begin{thm}\label{thm:large-HJ}
%Assume that {\rm(A1)--(A7)} hold. 
%Let $v$ the viscosity solution to \eqref{eq:CL2}. 
%There exists $(V,d)\in C(\T)\times\R$ such that 
%\[
%\|v(\cdot,t)+f(0)t-V(\cdot-dt)\|_{\Li(\T)}\to0
%\quad\text{as} \ t\to\infty. 
%\]
%Moreover, $V$ satisfies the condition 
%\[
%a(y_2-y_1)\le V(y_2)-V(y_1)\le b(y_2-y_1) 
%\quad\text{for} \ y_1, y_2\in\R 
%\ \text{with} \ y_1<y_2, 
%\]
%where $a, b$ are given by \eqref{def:I}, and $d$ is given by $f(u)=du$ for $u\in I=[a,b]$. 
%In case $I=\{0\}$, we have $V(x)\equiv \overline{V}$ for some $\overline{V}\in\R$. 
%\end{thm}

\begin{proof}[Proof of Theorem {\rm\ref{thm:large-HJ}}]
Let $v$ be the viscosity solution to \ref{VHJ}. 
Set $\tilde{v}(x,t):= v(x,t) + f(0)t$ for $x\in\T$, $t\ge0$. Then,  
\begin{equation}\label{eq:HJ2}
\partial_{t} \tilde{v} + f(\partial_{x} \tilde{v}) - f(0) = \beta'(\partial_{x} \tilde{v})\partial^2_{x} \tilde{v}
\quad \text{in} \ \T\times(0,\infty).
\end{equation}  
%If $v$ is a constant function, then $\tilde{v}$ is a viscosity solution to (\ref{HJ-2}). 
Let $u := \partial_{x}v$. 
In light of Theorem \ref{thm:main}, $u$ is the entropy solution to \ref{CL} with $u_{0} = \partial_{x} v_{0} \in L^{\infty}(\T)$.  
By Theorem \ref{thm:panov}, there exists $(U,d) \in L^{\infty}(\T)\times\R$ such that 
\begin{equation}\label{conv:L1}
{\rm ess}\lim_{t \to \infty}(u(t,\cdot) - U(\cdot-dt)) = 0 \ \ \mbox{ in } L^{1}(\T)%\ \ \mbox{ as } t \to \infty, 
\end{equation}
with $\int_{\T}U(y)\, dy = \int_{\T}u_0(y)\, dy =\int_{\T}\partial_x v_0(y)\,dy = 0$. 

Let $I=[a,b]$ be the set defined by \eqref{def:I}. 
In case $a=b=0$, we set $\tilde{V}(x)\equiv0$. 
In case $a<0<b$, 
%we set $\tilde{V}(x):=\int_0^x U(y)\,dy$. 
%Note that
 by Theorem \ref{thm:panov}
$f(u) - du$ and $\beta(u)$ are constant on $[ \mbox{{\rm ess}}\inf_{\T} U,  \mbox{{\rm ess}}\sup_{\T} U]$, 
and $0\in I$, which implies that 
$f(u) - du\equiv f(0)$ and $\beta(u)\equiv\beta(0)$ 
on $[ \mbox{{\rm ess}}\inf_{\T} U,  \mbox{{\rm ess}}\sup_{\T} U]$.  
Moreover, since $\beta$ is constant in a neighborhood of $u=0$, 
$\beta'(u)=0$ on this set. 
Here, $U(x-dt)$ is a weak solution to the linear equation $\partial_{t}u + d \partial_{x} u = 0$, and therefore it is the entropy solution to the linear equation.  
Therefore, 
we see that $\tilde{V}(x-dt) := \int_{0}^{x}U(y-dt)dy - d \int_{0}^{t}U(-d\tau)d\tau$ is a viscosity solution to \eqref{eq:HJ2} with $v_0=%\tilde{V}
\int_{0}^{x}U(y)dy$ by Theorem \ref{thm:main}. 
Setting 
\[
m(t) := \min_{x \in \T} \big(\tilde{v}(x,t) - \tilde{V}(x - dt)\big), 
\quad
M(t) := \max_{x \in \T} \big(\tilde{v}(x,t) - \tilde{V}(x - dt)\big), 
\]
we see that $m$ is nondecreasing, and $M$ is nonincreasing  by the comparison principle for \eqref{eq:HJ2}. 
Hence, the limits $m_{\infty}: = \lim_{t \to \infty} m(t)$, $M_{\infty} := \lim_{t \to \infty} M(t)$ exist.  
Moreover, letting $x_1, x_2\in\T$ so that 
$M(t)=\tilde{v}(x_1,t) - \tilde{V}(x_1-dt)$ and $m(t)=\tilde{v}(x_2,t) - \tilde{V}(x_2-dt)$, we have 
\begin{align*}
&M(t)-m(t)
=
\int_{x_2}^{x_1}\frac{\partial}{\partial x}\big(\tilde{v}(y,t)-\tilde{V}(y-dt)\big)\,dy
\le\int_{x_2}^{x_1}\left|\frac{\partial}{\partial x}\big(\tilde{v}(y,t)-\tilde{V}(y-dt)\big)\right|\,dy
 \\
\le& 
\int_{\T}\left|\frac{\partial}{\partial x}\big(\tilde{v}(y,t)-\tilde{V}(y-dt)\big)\right|\,dy
= 
\int_{\T}\left|u(y,t)-U(y-dt)\right|\,dy. 
\end{align*}
Therefore, in light of \eqref{conv:L1} and the continuity of $M(t)$ and $m(t)$, 
we obtain $M_\infty=m_\infty =: m_{\ast}$, which implies  
\[
\|v(\cdot,t)+f(0)t-V(\cdot-dt)\|_{\Li(\T)}\to0
\quad\text{as} \ t\to\infty, 
\]
where $V(x - dt) = \tilde{V}(x - dt) + m_{\ast}$. 

Finally noting that $[ \mbox{{\rm ess}}\inf_{\T} U,  \mbox{{\rm ess}}\sup_{\T} U]\subset[a,b]$, and $ V(y_2)-V(y_1)=\int_{y_1}^{y_2}U\,dy$, 
we have 
\[
a(y_2-y_1)\le V(y_2)-V(y_1) \le b(y_2-y_1)
\quad\text{for all} \ y_1<y_2. 
\]
\end{proof}

\subsection{Large-time behavior for \ref{CL}}
In this subsection we assume (A8), 
therefore \ref{CL} turns out to be the equation of the form
\begin{equation}\label{eq:CL2}
\partial_{t} u + \partial_{x} f(x,u) = \partial_{x} ( \alpha(x) \partial_{x}u )
\quad \text{in} \ \T\times(0,\infty).
\end{equation}
We recall a result on the large time behavior for degenerate viscous Hamilton--Jacobi equations in \cite{CGMT}. 
\begin{thm}[{\cite[Theorem 1.1]{CGMT}}]\label{thm:CGMT}
Assume that {\rm (A1)--(A4), (A5)', (A6), (A8)} hold. 
Let $v$ be the viscosity solution to {\rm \ref{HJ}}. 
There exists $(V,c) \in \Lip(\T)\times\R$ such that 
\begin{equation*}
\lim_{t \to + \infty} \|v(\cdot,t) - \big(V-ct)\|_{\Li(\T)}= 0,  
%\ \ \text{ in } C(\T). 
\end{equation*}
where $(V,c)$ is a viscosity solution of the ergodic problem $f(x, \partial_{x}V) = \alpha(x) \partial_{x}^{2} V + c$ in $\T$. 
\end{thm}

%\begin{thm}\label{thm:large-CL}
%Assume that {\rm(A1)--(A6), (A8)} hold. 
%Let $u$ be the entropy solution to {\rm \ref{CL}}. There exists $(U,c) \in L^{\infty}(\T)\times\R$ such that 
%\begin{equation*}
%\mbox{{\rm ess}} \lim_{t \to + \infty} \|u(\cdot,t) - (U-ct)\|_{L^1(\T)} = 0. 
%\end{equation*}
%\end{thm}
\begin{proof}[Proof of Theorem {\rm\ref{thm:large-CL}}]
Let $u$ be the entropy solution to \eqref{eq:CL2}. 
By Theorem \ref{thm:main} setting $v(x,t):=\int_0^xu(y,t)\,dy 
+ \widehat{C}(t)$ is the viscosity solution to 
\ref{HJ} with $v_0:=\int_0^xu_0\,dy$. 
Therefore, due to Theorem \ref{thm:CGMT}, 
there exists $(V,c) \in \Lip(\T)\times\R$ which satisfies 
$f(x, \partial_{x}V) = \alpha(x) \partial_{x}^{2} V + c$ in $\T$ 
in the sense of viscosity solutions, and 
\begin{equation}\label{conv:int}
\int_0^xu(y,t)\,dy + \widehat{C}(t) -(V(x)-ct)\to0 \ \ \mbox{ as } t \to \infty \ \text{ in} \ C(\T). 
\end{equation}
Take $U\in L^{\infty}(\T)$ so that $V(x)=\int_0^xU(y)\,dy$. 
Then, 
$\int_0^x (u(y,t)-U(y))\,dy + \widehat{C}(t)+ 
ct\to0$ uniformly for $x\in\T$ as $t\to\infty$, which implies 
\[
\int_{\T} \left( u(y,t) + \widehat{C}(t) -\big(U(y)-ct\big) \right) \,dy\to0 \quad\text{as} \ t\to\infty. 
\]
\end{proof}
If $f$ is independent of $x$, then we have $L^1$ convergence in Theorem \ref{thm:large-CL}. More precisely, we obtain 
\begin{prop}
Assume that {\rm(A1)--(A4), (A5)', (A6)} and 
\begin{enumerate}
\item[{\rm(A8)'}] $f(x,u) \equiv f(u) \in C^2(\R)$, $u\mapsto f(u)$ is strictly convex, $\beta(u)=0$ for all %$(x,u)\in\T\times\R$
$u \in \R$ 
\end{enumerate}
hold. 
Let $u$ be the entropy solution to {\rm \ref{CL}}. There exists $U \in L^{\infty}(\T)$ such that 
\begin{equation*}
{\rm ess }\lim_{t \to + \infty} \int_{\T} |u(y,t) - U(y)| \,dy 
= 0. 
\end{equation*}
\end{prop}
\begin{proof}
By (\ref{conv:int}), we see that 
\begin{equation*}
\int_0^x u(y,t)\,dy + \widehat{C}(t) - \left( \int_{0}^{x}U(\xi)\,d\xi + V(0)-ct \right) \to 0 \ \ \mbox{ as } t \to \infty, 
\end{equation*}
for all $x \in (0,1)$. 
From this it immediately follows that 
\begin{equation}\label{conv:int-2}
\int_{x_{1}}^{x_{2}} (u(y,t) - U(y))\, dy \to 0 \ \ \mbox{ as } t \to \infty  
\end{equation}
for any $x_{1}, x_{2} \in (0,1)$. 

On the other hand, due to the independence of $x$ of $f$, 
we have  
\begin{equation*}
\| u(\cdot,t) \|_{L^{\infty}(x_{1},x_{2})} \le \| u(\cdot,t) \|_{L^{\infty}(\R)} \le C,
\ \text{and} \  
\int_{x_{1}}^{x_{2}} |\partial_{x}u|\, dx \le \int_{\R}|\partial_{x}u_{0}|\,dx 
\end{equation*}
for all $x_{1}, x_{2} \in (0,1)$ with $x_{1} < x_{2}$. 
Let $\{t_{n}\}$ be a sequence of the Lebesgue points of $u(x,\cdot)$ satisfying $t_{n} \to \infty$ as $n \to \infty$. 
By Helly's selection theorem, there exists a subsequence $\{t_{n_{k}}\}$ and a function $\overline{u} \in BV(x_{1},x_{2})$ such that 
\begin{equation}\label{eq:sub-u}
u(x, t_{n_{k}}) \to \overline{u}(x)\ \ \mbox{ a.e. in } (x_{1},x_{2})\ \ \mbox{ as } k \to \infty. 
\end{equation}
%Hence it follows that 
%\begin{equation*}
%u(x, t_{n_{k}}) - U(x) \to \overline{u}(x) - U(x)\ \ \mbox{ a.e. in } (x_{1},x_{2}) \ \ \mbox{ as } k \to \infty. 
%\end{equation*}
By the Lebesgue dominated convergence theorem, we see that 
$u(\cdot, t_{n_k})\to \overline{u}$ in $L^1(x_1,x_2)$ as $k\to\infty$ 
for all $(x_1, x_2)\subset(0,1)$, 
and 
\begin{equation}\label{conv:int-3}
%\int_{x_{1}}^{x_{2}} |u(x, t_{n_{k}}) - \overline{u}(x)| dx \to 0
\int_{x_{1}}^{x_{2}} u(x,t_{n_{k}})dx \to \int_{x_{1}}^{x_{2}} \overline{u}(x) dx \ \ \mbox{ as } k \to \infty. 
\end{equation}
Moreover, by (\ref{conv:int-2}), (\ref{conv:int-3}), it is deduced that 
\begin{equation*}
\int_{x_{1}}^{x_{2}} \overline{u}(x)\, dx = \int_{x_{1}}^{x_{2}} U(x)\, dx \ \ \mbox{ for all } (x_{1},x_{2}) \subset (0,1),  
\end{equation*}
which implies that $\overline{u}(x) = U(x)$ a.e. in $(0,1)$. 
Therefore, we obtain $u(\cdot, t_{n_k})\to U$ in $L^1(x_1,x_2)$ as $k\to\infty$. 
Since this holds for any subsequnece satisfying \eqref{eq:sub-u}, 
we obtain the conclusion. 
\end{proof}

%\subsection*{Acknowledgments}
% 

%%%%%%%%%%%%%%%%%%%% References %%%%%%%%%%%%%%%%%%%%%%%%%%

%{\color{red}
%Referenceの整理・追加（短めの論文なのでReferenceは１５程度に抑える？
%HJとCLのバランスを考える．
%}


\begin{thebibliography}{99}

%\bibitem{AFP}
%L. Ambrosio, N. Fusco and D. Pallara, 
%\emph{Functions of Bounded Variation and Free Discontinuity Problems}, 
%Oxford Science Publications, (2000).

%\bibitem{AM}
%B. Andreianov, M. Maliki, 
%\emph{A note on uniqueness of entropy solutions to degenerate parabolic equations in $\R^{N}$},
%\newblock NoDEA Nonlinear Differential Equations Appl., \textbf{17} (2010), no.1, 109-118.
%%https://doi.org/10.1007/s00030-009-0042-9

%\bibitem{CG}
%\newblock C. Canc\'{e}s, T. Gallou\"{e}t,
%\newblock \emph{On the time continuity of entropy solutions}, 
%\newblock J. Evol. Equ. \textbf{11} (2011), no.1, 43--55.

%\bibitem{C99} 
%\newblock J. Carrillo, 
%\newblock \emph{Entropy solutions for nonlinear degenerate problems},
%\newblock Arch. Rational. Anal., \textbf{147} (1999), 269-361.

%\bibitem{B}
%P. Bernard, 
%\emph{The asymptotic behaviour of solutions of the forced Burgers equation on the circle}, 
%Nonlinearity 18 (2005), no. 1, 101--124. 

\bibitem{CGMT} 
F. Cagnetti, D. Gomes, H. Mitake, H. V. Tran, 
\emph{A new method for large time behavior of degenerate viscous Hamilton-Jacobi equations with convex Hamiltonians}, 
Ann. Inst. H. Poincar\'e C Anal. Non Lin\'eaire \textbf{32} (2015), no. 1, 183--200.

\bibitem{C}
V. Caselles, 
\emph{Scalar conservation laws and Hamilton-Jacobi equations in one-space variable}, 
Nonlinear Anal. \textbf{18} (1992), no. 5, 461--469. 

\bibitem{CK}
\newblock G. Q. Chen, K. H. Karlsen, 
\newblock \emph{Quasilinear anisotropic degenerate parabolic equations with time-space dependent diffusion coefficients},
\newblock Commun. Pure Appl. Anal., \textbf{4} (2005), 241--266.

\bibitem{CP1}
\newblock G. Q. Chen, B. Perthame, 
\newblock \emph{Well-posedness for non-isotropic degenerate parabolic-hyperbolic equations}, 
\newblock Ann. Inst. H. Poincar\'{e} Anal. Non Lin\'{e}aire, \textbf{20}(4) (2003), 645--668.
%https://doi.org/10.1016/S0294-1449(02)00014-8

\bibitem{CP2}
\newblock G. Q. Chen, B. Perthame, 
\newblock \emph{Large-time behavior of periodic entropy solutions to anisotropic degenerate parabolic-hyperbolic equations}, 
\newblock Proc. Amer. Math. Soc. \textbf{137} (2009), no. 9, 3003--3011. 

\bibitem{Coc}
\newblock G. M. Coclite,
\newblock "Scalar conservation laws", 
\newblock SpringerBriefs Math., Springer, (2024). 


\bibitem{CFN}
L. Corrias, M. Falcone, R. Natalini, 
\emph{Numerical schemes for conservation laws via Hamilton-Jacobi equations}, 
Math. Comp. 64 (1995), no. 210, 555--580. 

\bibitem{CIL}
M. G. Crandall, H. Ishii, P.-L. Lions, 
\emph{User's guide to viscosity solutions of second order partial differential equations}, 
Bull. Amer. Math. Soc. (N.S.) 27 (1992), no. 1, 1--67. 

%\bibitem{D}
%\newblock C. M. Dafermos, 
%\newblock "Hyperbolic conservation laws in continuum physics. Third edition", 
%\newblock Grundlehren der Mathematischen Wissenschaften, 325. Springer-Verlag, Berlin, (2010).

%\bibitem{EG}
%\newblock L. C. Evans and R. Gariepy, 
%\newblock "Measure theory and fine properties of functions", 
%\newblock Studies in Advanced Math., CRC Press, London, (1992).

%\bibitem{KKR}
%\newblock K. H. Karlsen, U. Koley, N. H. Risebro, 
%\newblock \emph{An error estimate for the finite difference approximation to degenerate convection-diffusion equations}, 
%\newblock Numer. Math. \textbf{121} (2012), no. 2, 367--395.
%https://doi.org/10.1007/s00211-011-0433-9

%{\color{red}
%\bibitem{KRT}
%\newblock K. H. Karlsen, N. H. Risebro, J. D. Towers, 
%\newblock \emph{On a nonlinear degenerate parabolic transport-diffusion equation with a discontinuous coefficient}, 
%\newblock Electron. J. Differential Equations (2002), No. 93, 23 pp.
%どこかに引用するか，引用から外すか．
%%https://ejde.math.txstate.edu/
%}


\bibitem{GMT}
D. A. Gomes, H. Mitake, H. V. Tran, 
\emph{The large time profile for Hamilton-Jacobi-Bellman equations}, 
Math. Ann. 384 (2022), no. 3-4, 1409--1459.

\bibitem{KU}
\newblock K. H. Karlsen, S. Ulusoy, 
\newblock \emph{On a hyperbolic Keller-Segel system with degenerate nonlinear fractional diffusion}, 
\newblock Netw. Heterog. Media \textbf{11} (2016), no. 1, 181--201.


%\bibitem{KR03}
%\newblock K. H. Karlsen, N. H. Risebro,
%\newblock \emph{On the uniqueness and stability of entropy solutions of nonlinear degenerate parabolic equations with rough coefficients},
%\newblock Discrete Contin. Dyn., \textbf{9}(5) (2003), 1081--1104.

\bibitem{Kr} 
\newblock S. N. Kru\v{z}kov, 
\newblock \emph{First order quasilinear equations in several independent variables}, 
\newblock Math. USSR Sbornik, \textbf{10} (1970), 217-243.
%https://doi.org/10.1070/SM1970v010n02ABEH002156

\bibitem{LSU}
\newblock O. A. Lady\v{z}enskaja, V. A. Solonnikov, N. N. Ural'ceva,
\newblock "Linear and quasilinear equations of parabolic type",
\newblock American Mathematical Society, Providence, R.I., (1967).


%\bibitem{L-book}
%P. L. Lions, 
%\emph{Generalized Solutions of Hamilton-Jacobi Equations}, 
%Research Notes in Mathematics, Vol. 69, Pitman, Boston, 1982. 

\bibitem{MNRR}
\newblock  J. M\'alek, J. Ne\v{c}as, M. Rokyta, M. R\r{u}\v{z}i\v{c}ka, 
\newblock "Weak and measure-valued solutions to evolutionary PDEs", 
\newblock Applied Mathematics and Mathematical Computation, 13. Chapman \& Hall, London, (1996).

%{\color{red}
\bibitem{MT}
\newblock M. Maliki, H. Tour\'{e}, 
\newblock \emph{
%Uniqueness of entropy solutions for nonlinear degenerate parabolic problems
Solution g\'{e}n\'{e}ralis\'{e}e locale d'une \'{e}quation parabolique quasi lin\'{e}aire d\'{e}g\'{e}n\'{e}r\'{e}e du second ordre}, 
\newblock 
%J. Evol. Equ. \textbf{3} (2003), no.4, 603-622. 
Ann. Fac. Sci. Toulouse Math. (6)\textbf{7} (1998), no.1, 113--133.
%https://doi.org/10.1007/978-3-0348-7924-8\_31
%http://www.numdam.org/item/AFST_1998_6_7_1_113_0/
%どこかに引用するか，引用から外すか．
%}

%\bibitem{OK}
%\newblock O .A. Oleinik and S. N. Kru\v{z}kov
%\newblock \emph{Quasi-linear second-order parabolic equations with many independent variables}, 
%\newblock Russian Math. Surveys \textbf{16} (1961), no.5, 105--146.

\bibitem{P19}
\newblock E. Yu. Panov, 
\newblock \emph{The long time behavior of periodic entropy solutions to degenerate nonlinear parabolic equations}, 
\newblock J. Math. Sci. (N.Y.) \textbf{242} Problems in mathematical analysis. No. 99 (2019), no.2, 308--322.

\bibitem{P20}
\newblock E. Yu. Panov, 
\newblock \emph{On almost periodic viscosity solutions to Hamilton-Jacobi equations}, 
\newblock Minimax Theory Appl. \textbf{5} (2020), no.2, 383--400.

%\bibitem{VH} 
%\newblock A. I. Vol'pert and S. I. Hudjaev, 
%\newblock \emph{Cauchy's problem for degenerate second order quasi-linear parabolic equations}, 
%\newblock Math. USSR-Sb., \textbf{7} (1969), 365--387.


\bibitem{W17}
\newblock H. Watanabe,
\newblock \emph{Strongly degenerate parabolic equations with variable coefficients}, 
\newblock Adv. Math. Sci. Appl., \textbf{26} (2017), 143--173.


\bibitem{W21}
H. Watanabe, 
\emph{Traveling waves to one-dimensional Cauchy problems for scalar parabolic-hyperbolic conservation laws}, 
J. Differential Equations 286 (2021), 474--493. 

%\bibitem{XYJ}
%\newblock T. Xu, J. Yin and S. Ji, 
%\newblock \emph{Discontinuous traveling wave entropy solutions for a sedimentation-consolidation model},
%\newblock Nonlinear Anal. \textbf{159} (2017), 468--481. 
%https://doi.org/10.1016/j.na.2016.11.021

%\bibitem{Z} 
%\newblock W. P. Ziemer, 
%\newblock "Weakly differentiable functions", 
%\newblock Springer-Verlag, New York, (1989).
\end{thebibliography}
\end{document}